\documentclass[10pt]{amsart}
\usepackage{xypic}
\usepackage[pdftex,colorlinks=true,citecolor=blue,linkcolor=blue]{hyperref}
\usepackage{amssymb}
\newtheorem{theorem}{Theorem}[section]
\newtheorem{lemma}[theorem]{Lemma}

\newtheorem{definition}[theorem]{Definition}

\newtheorem{cnj}[theorem]{Conjecture}

\newcommand{\Z}{\mathbb{Z}}
\newcommand{\bH}{\mathbb{H}}
\newcommand{\Q}{\mathbb{Q}}
\newcommand{\R}{\mathbb{R}}
\newcommand{\F}{\mathbb{F}}
\newcommand{\PP}{\mathbb{P}}
\newcommand{\HH}{{\mathcal H}}
\newcommand{\cF}{\mathcal{F}}
\newcommand{\FF}{\cF_c}
\newcommand{\frob}{Frob}
\newcommand{\too}[1]{\mathop{\longrightarrow}\limits^{#1}}

\DeclareMathOperator{\GL}{GL}
\DeclareMathOperator{\SL}{SL}
\DeclareMathOperator{\diag}{diag}
\DeclareMathOperator{\Stab}{Stab}
\DeclareMathOperator{\Ind}{Ind}

\DeclareMathOperator{\tr}{Tr}
\DeclareMathOperator{\Hom}{Hom}

\newcommand{\isom}{\cong}
\date{May 14, 2012}
\title{Reducible Galois representations and the homology of $\GL(3,\Z)$}
\author{Avner Ash}
\address{Boston College, Chestnut Hill, MA  02445}
\email{Avner.Ash@bc.edu}
\author{Darrin Doud}
\address{Brigham Young University, Provo, UT  84602}
\email{doud@math.byu.edu}
\thanks{The first author thanks the NSA for support of this research through NSA grant H98230-09-1-0050.  This manuscript is submitted for publication with the understanding that the United States government is authorized to reproduce and distribute reprints.}

\begin{document}
\begin{abstract} Let $\bar\F_p$ be an algebraic closure of a finite field of characteristic $p$.  Let $\rho$ be a continuous homomorphism from the absolute Galois group of $\Q$ to $\GL(3,\bar\F_p)$ which is isomorphic to a direct sum of a character and a two-dimensional odd irreducible representation. Under the condition that the conductor of $\rho$ is squarefree, we prove that $\rho$ is attached to a Hecke eigenclass in the homology of an arithmetic subgroup $\Gamma$ of $\GL(3,\Z)$.  In addition, we prove that the coefficient module needed is, in fact, predicted by the main conjecture of \cite{ADP}.\end{abstract}
\maketitle

\section{Introduction}\label{intro}

Fix a prime $p$ and let $\bar\F_p$ be an algebraic closure of a finite field of characteristic $p$.  Generalizations of Serre's conjecture~\cite{S} connect the homology of an arithmetic group with a Galois representation $\rho$.  When the target of $\rho$ is $\GL(n,\bar\F_p)$, a conjecture was first published in~\cite{AS}, which was extended in~\cite{ADP} and further improved in~\cite{H}.  See Section~\ref{connections} for definitions and a statement of the conjecture in the cases that we will require.

In this paper, we will first discuss a general method of computing Hecke operators acting on the homology of certain complexes.  This should help us attach Galois representations to the boundary homology of congruence subgroups of $\GL(n,\Z)$ for any $n$.  In this paper we are able to treat the case of $n=3$. We use this method to prove the existence of eigenclasses in the homology of arithmetic groups that are attached (see section~\ref{connections}) to Galois representations that decompose as the sum of an irreducible two-dimensional odd Galois representation and a character.  

If we let $X_n$ be the Borel-Serre bordification of the symmetric space for $\GL(n,\R)$, the representations that we consider should be attached to eigenclasses in the homology of the boundary of $X_n/\Gamma$ for $n=3$. However, when $n>2$, the mod $p$ topology of the boundary is quite complicated.  Our method is designed to avoid the need to compute the homology of this boundary. 

Our main theorem (Theorem~\ref{final}) will show that for any $\rho:G_\Q\to\GL(3,\bar\F_p)$ with $p>3$ prime, such that $\rho$ is a direct sum of a two-dimensional odd irreducible representation and a character, and the Serre conductor of $\rho$ is squarefree, there is a Hecke eigenclass in the homology $H_3(\Gamma_0(3,N),V\otimes\epsilon)$ that is attached to $\rho$, where $N$, $V$, and $\epsilon$ are predicted by Conjecture~\ref{conjecture}.  In addition, for any such $\rho$, there may be several possible values for $V$; we will show that all of them yield eigenclasses with $\rho$ attached.

\section{Conjectural connections between Galois representations and arithmetic homology}\label{connections}

A Hecke pair $(\Gamma,S)$ in a group $G$ is a subgroup $\Gamma\subset G$ and a subsemigroup $S\subset G$ such that $\Gamma\subset S$ and for any $s\in S$ both $\Gamma\cap s^{-1}\Gamma s$ and $s\Gamma s^{-1}\cap \Gamma$ have finite index in $\Gamma$.

If we let $R$ be a ring and $M$ a right $R[S]$-module, then for $s\in S$ there is a natural action of a double coset $\Gamma s\Gamma$ on the homology $H_i(\Gamma,M)$.  We denote by $\HH(\Gamma,S)$ the $R$-algebra under convolution generated by all the double cosets $\Gamma s\Gamma$ with $s\in S$.  We call $\HH(\Gamma,S)$ a Hecke algebra, and the double cosets Hecke operators.  The action of the double cosets on homology makes $H_i(\Gamma,M)$ an $\HH(\Gamma,S)$-module.

We will use the following groups and semigroups in $\GL_n$.

\begin{definition} Let $N$ be a positive integer and $p$ a prime.
\begin{enumerate}
\item $S_0(n,N)^{\pm}$ is the semigroup of matrices $s\in M_n(\Z)$ such that $\det(s)$ is relatively prime to $pN$ and the first row of $s$ is congruent to $(*,0,\ldots,0)$ modulo $N$.
\item $S_0(n,N)$ is the subsemigroup of $s\in S_0(n,N)^{\pm}$ such that $\det(s)>0$.
\item $\Gamma_0(n,N)^\pm=S_0(n,N)^{\pm}\cap\GL(n,\Z)$.
\item $\Gamma_0(n,N)=S_0(n,N)\cap\GL(n,\Z)$.
\end{enumerate}
\end{definition}

In the case in which we are interested, the ring $R$ will be the algebraic closure $\bar\F_p$ of a finite field of order $p$, $S$  will be $S_0(n,N)$, and $\Gamma$ will be $\Gamma_0(n,N)$. We will denote the Hecke algebra $\HH(\Gamma_0(n,N),S_0(n,N))$ by $\HH_{n,N}$.  We note that $\HH_{n,N}$ is commutative, and contains the Hecke operators $T(\ell,k)=\Gamma D_{\ell,k}\Gamma$ where
$$D_{\ell,k}=\diag(\underbrace{1,\cdots,1}_{n-k},\underbrace{\ell,\cdots,\ell}_k).$$

\begin{definition} Let $V$ be an $\HH_{n,N}$-module, and let $v\in V$ be a simultaneous eigenvector of all the $T(\ell,k)$ with $\ell\nmid pN$ and $0\leq k\leq n$.  Denote by $a(\ell,k)$ the eigenvalue of $T(\ell,k)$ acting on $v$.

We say that the Galois representation $\rho:G_\Q\to\GL(n,\bar\F_p)$ is attached to $v$ if
$$\det(I-\rho(\frob_\ell)X)=\sum_{k=0}^n(-1)^k\ell^{k(k-1)/2}a(\ell,k)X^k$$
for all prime $\ell\nmid pN$ for which $\rho$ is unramified at $\ell$.
\end{definition}
\noindent(Note that we use the arithmetic Frobenius, so that if $\omega$ is the cyclotomic character, $\omega(\frob_\ell)=\ell$.)

The $\HH_{n,N}$ modules that we use to find Hecke eigenvectors attached to Galois representations will be computed as the homology of admissible modules for $S_0(n,N)$.

\begin{definition} For a fixed prime $p$, an admissible module $M$ (cf.~\cite{Ash-Duke}) for a subsemigroup $S$ of $\GL(n,\Q)$ is a finite dimensional vector space over $\F_p$ on which there exists an $N$ such that $S$ acts on $M$ through reduction modulo $N$.   (In particular, the denominators in $S$ are all prime to $N$.)
\end{definition}

Specifically, the admissible modules that we need will be irreducible $\bar\F_p[\GL(3,\F_p)]$-modules, on which subsets of $\GL(3,\Z)$ consisting of matrices of determinant prime to $p$ will act through reduction modulo $p$.  Such modules are parameterized by triples $(a,b,c)$ with $0\leq a-b,b-c\leq p-1$ and $0\leq c\leq p-2$.  The module corresponding to the triple $(a,b,c)$ will be denoted $F(a,b,c)$. (See Section~\ref{irreducible} for details.)

In order to connect irreducible modules with Galois representations, we will need the mod $p$ cyclotomic character $\omega$, and the {\it niveau} two characters $\omega_2$ and $\omega_2'$ \cite{S72}.  Note that a power of the cyclotomic character may be written as $\omega^a$, with $a$ well defined modulo $p-1$ (since $\omega$ has order $p-1$).  Similarly, a power of $\omega_2$ that is not a power of $\omega$ may be written as $\omega_2^m$ with $m$ well defined modulo $p^2-1$ and not a multiple of $p+1$.  Given such an $m$, with $0\leq m<p^2-1$, we may write $m=a+bp$  with $0\leq a,b<p$ and $a\neq b$ (for instance writing $m$ in base $p$). In fact, we may write $m=a+bp$ with $0<a-b\leq p$ (by adding $p$ to $a$ and subtracting one from $b$, as necessary).  Note that $a$ and $b$ are only well defined modulo $p-1$, since adding $p-1$ to each of $a$ and $b$ changes $m$ by $p^2-1$, yielding the same power of $\omega_2$.  

Finally, we recall the two different types of wild ramification for a representation $\rho:I_p\to\GL_2(\bar\F_p)$ of the form
$$\rho(I_p)\sim\begin{pmatrix}\omega^{a+1}&*\cr0&\omega^a\end{pmatrix},$$
namely {\it peu ramifi\'e} and {\it tr\`es ramifi\'e}, and refer the reader to \cite{S} for definitions.

We now state a conjecture connecting certain three-dimensional Galois representations with eigenvectors in arithmetic cohomology groups.  We state the conjecture only for certain  representations; for a more general conjecture that applies to a much wider class of representations see \cite{ADP,AS,H}.

\begin{cnj}\label{conjecture} Let $p$ be a prime, and let $\bar\F_p$ be an algebraic closure of $\F_p$. Let $\rho:G_\Q\to\GL(3,\bar\F_p)$ be a Galois representation that is a sum of an irreducible odd two-dimensional representation and a character.  Let $N$ be the Serre conductor of $\rho$ and $\epsilon$ the nebentype of $\rho$ (see \cite{S}).  Then we may choose an irreducible admissible $F_p[\GL(3,\F_p)]$-module $V$ such that $\rho$ is attached to a cohomology class in the $\HH_{3,N}$-module $H_3(\Gamma_0(3,N),V\otimes\epsilon)$.

If $\rho=\sigma\oplus\omega^c\psi$, where $\psi$ has conductor prime to $p$, we may describe the possible $V$ in terms of the restriction of $\sigma$ to inertia at $p$.  If
$$\sigma|_{I_p}\sim\begin{pmatrix}\omega^a&*\cr0&\omega^b\end{pmatrix}$$
we choose $a,b,c$ modulo $p-1$ so that $0< a-b,b-c\leq p$ and $0\leq c<p-1$, with the restriction that if $\sigma$ is {\it tr\`es ramifi\'e}, then $a-b=p$ and let $V$ be the irreducible module $F(a-2,b-1,c)$.  We may also choose $a,b,c$ so that $0<c-a,a-b\leq p$ and $0\leq b<p-1$, with $a-b=p$ if $\sigma$ is {\it tr\`es ramifi\'e}, and let $V=F(c-2,a-1,b)$.

If
$$\sigma|_{I_p}\sim\begin{pmatrix}\omega_2^{a+bp}&0\cr0&\omega_2'^{a+bp}\end{pmatrix}$$
and $0<a-b\leq p$, we either take $0<a-b,b-c\leq p$ and $0\leq c<p-1$ and choose $V=F(a-2,b-1,c)$, or we take $0<c-a,a-b\leq p$ and $0\leq b<p-1$ and take $V$ to be $F(c-2,a-1,b)$.

For each value of $V$ described above, there is an $\HH_{3,N}$-eigenclass in $H_3(\Gamma_0(3,N),V\otimes\epsilon)$ with $\rho$ attached.
\end{cnj}

Our goal in this paper is to prove the following theorem.
\begin{theorem}\label{final}
Let $p>3$.  Then Conjecture~\ref{conjecture} is true for representations $\rho$ having squarefree conductor.
\end{theorem}

Note that generically, for a tamely ramified Galois representation, there will be two choices of the integers $a$ and $b$ (obtained by permuting the diagonal characters in each case).  Hence, for a tamely ramified representation there will normally be four predicted weights (if $\sigma$ is wildly ramified, there will only be two weights, since we cannot permute the two characters on the diagonal). It can happen that there are additional weights. For instance, if $a-b\equiv 1\pmod{p-1}$, we may choose $a=b+1$ or $a=b+p$.  Unless otherwise indicated (i.e. in the {\it tr\`es ramifi\'e} case), all of these weights are predicted.

In the conjecture, the two predictions of weights arise from embedding the image of $\rho$ into one of the two standard Levi subgroups
$$\begin{pmatrix}*&*&0\cr *&*&0\cr0&0&*\end{pmatrix}\quad\text{ and }\quad \begin{pmatrix}*&0&0\cr 0&*&*\cr0&*&*\end{pmatrix}.$$
There is also an embedding of the image of $\rho$ into the Levi subgroup
$$\begin{pmatrix}*&0&*\cr0&*&0\cr*&0&*\end{pmatrix}$$
but this embedding, for $\rho$ of the type we consider, would violate a strict parity condition, and thus has no predicted weights according to the main conjecture of \cite{ADP}.  Theoretical considerations lead us to believe that we might find an eigenclass attached to $\rho$ in a weight predicted from this forbidden embedding in a homology group of different degree, namely, $H_2(\Gamma_0(3,N),V\otimes\epsilon)$.  Proving the existence of this class is one possible future application of the techniques described in this paper.

\section{Hecke actions on induced representations}\label{g-sets}

  For any set $A$, $\Z[A]$ denotes the abelian group of formal finite linear combinations of elements of $A$.  If $A$ is a semigroup or group, $\Z[A]$ is naturally a ring.   Tensor products without a subscript are to be taken over $\Z$.

We follow Brown's notation for tensor products  p. 55, Chap. III~\cite{B}, except our conventions reverse left and right.  
If $G$ is a group, $H$ a subgroup of $G$, $A$ is a right $\Z H$-module and $B$ is a left $\Z H$-module and a right $\Z G$-module, then $A\otimes_{\Z H} B$ denotes the right 
$\Z G$-module where $ah\otimes_{\Z H} b=a\otimes_{\Z H} hb$ and 
$(a\otimes_{\Z H} b)g = a\otimes_{\Z H} bg$ for any $h\in H$ and $g\in G$.
If $B=\Z G$ (with the obvious right action of $H$ and left action of $G$),  $A\otimes_{\Z H} \Z G$ is the induced module from $H$ to $G$ of $A$.

If $M$ and $N$ are two right $\Z G$-modules, then $M\otimes_G N$ denotes the right 
$\Z$-module where $m\otimes_{G} n=mg\otimes_{G} ng$.  It is equal to the coinvariants of the right $\Z G$-module $M\otimes N$ where
$(m\otimes n)=mg\otimes ng$.

Let $\Gamma\subset\ S \subset G$ where $G$ is a group, $\Gamma$ a subgroup, $S$ a subsemigroup and $(\Gamma,S)$ a Hecke pair.  Let $X$ be a set on which $G$ acts on the right.  

Remark: In general, the $S$-action will not preserve the $\Gamma$-orbits in $X$.  Here is an example:  Consider $\Gamma=\Gamma_0(2,25)$ acting on $P^1(\Z)$.  Let $s=\diag(2,1)$.  Note that $(5:1)$ and $(5:6)=(5:1)\bmatrix 1&1\\0&1\endbmatrix$ are in the same $\Gamma$-orbit.  However $(5:1)s=(10:1)$ and $(5:6)s=(10:6)=(5:3)$ are in different $\Gamma$ orbits.  

For any $x\in X$ and any subsemigroup $T$ of $G$, write $T_x=\Stab_T x$.  We define the concept of an $S$-sheaf $W$ on $X$:  For any 
$x\in X$, $W_x$ is an abelian group, and for any $g\in S, x\in X$ there is given a homomorphism $\mu(g):W_x\to W_{xg}$ such that $\mu(h)\circ\mu(g)=\mu(gh)$ for all $g,h\in S$ and $\mu(1)=id_{W_x}$.  Then $W=\oplus_{x\in X} W_x$ is a right $\Z S$-module.  If $F$ is any 
right $\Z S$-module, then $W\otimes F$ is a right $\Z S$-module with the action given by $(w\otimes f)g=wg\otimes fg$ for any $g\in S$.

By Corollary 5.4 p. 68 in~\cite{B}, an $S$-sheaf $W$ restricted to $\Gamma$ is a direct sum of induced $\Gamma$-modules.
For $a\in X$, write $X_a=a\Gamma$ for the orbit of $a$ under $\Gamma$.
Choose a subset $A\subset X$ such that $X=\coprod_{a\in A} X_a$. For each $a\in A$ let $W[a]=\oplus _{x\in X_a} W_{x}\subset W$.  Then $W[a]$ is naturally isomorphic as right $\Gamma$-module to the induced module $W_a\otimes_{\Z\Gamma_a} \Z\Gamma$.  Note that in general $S$ does not preserve $W[a]$.

We abbreviate $\otimes_{\Z\Gamma_a}$ by $\otimes_{a}$.  Then
$$
W\approx \oplus_{a\in A} W[a] \approx
\oplus_{a\in A} W_a \otimes_{a} \Z\Gamma.
$$

Let $F$ be a right $S$-module.    Then 
$W\otimes F$ is an $S$-module and therefore 
$W\otimes_\Gamma F=H_0(\Gamma, W\otimes F)$ has a natural action on it of the Hecke algebra  $\HH(\Gamma,S)$ which we want to compute.

If $Y$ is a set on which a group $K$ acts on the right, let $\FF(Y/K)$ denote the set of $K$-invariant functions whose support is a finite number of $K$-orbits.  If in addition a group $H$ acts on the left of $Y$, commuting with the $K$-action, let 
$\FF(H\backslash Y/K)$ denote the set of $H\times K$-invariant functions whose support is a finite number of $H\times K$-orbits.   

Given a Hecke pair $(\Gamma,S)$, the Hecke algebra $\HH=\HH(\Gamma,S)$ will be identified with $\FF(\Gamma\backslash S/\Gamma)$, where multiplication is convolution of functions.

For any $\Gamma$-module $N$, let $\Phi$ denote the natural projection from $N$ to the coinvariants $N_\Gamma=H_0(\Gamma,N)$.
If $h\in\HH$,
and 
$z\in W\otimes F$, 
writing $T_h$ for the Hecke operator corresponding to $h$ we have:
$$
\Phi(z)|T_h=\Phi\left({\sum_{g\in\Gamma\backslash S/\Gamma} h(g)zg}\right).
$$
This does not depend on the choice of coset representatives.
If $h$ is the characteristic function of 
$\Gamma s\Gamma=\coprod s_\alpha\Gamma$, then
$$
\Phi(z)|T_h=\Phi\left({\sum_{\alpha} zs_\alpha}\right)
$$
since $h(g)=0$ unless $g\in \Gamma s\Gamma$, in which case $g=s_\alpha\gamma$ for some $\alpha,\gamma$, and then $h(s_\alpha)=1$.

We will now determine the Hecke action in terms of the isomorphism of homology given by Shapiro's lemma.
Associativity of tensor products gives a canonical isomorphism 
$\lambda: W[a]\otimes_\Gamma F=(W_a\otimes_a\Z\Gamma)\otimes_\Gamma F\approx 
W_a\otimes_{\Gamma_a} F$ 
via $w\otimes_a\gamma\otimes_\Gamma f \mapsto w\otimes_{\Gamma_a}f\gamma^{-1}$.  

Now $h\in\HH$ acts on $W\otimes_\Gamma F = \oplus_a W[a]\otimes_\Gamma F\approx 
\oplus_a W_a\otimes_{\Gamma_a} F$.  We seek a formula for $T_h$ on an element in 
$\oplus_a W_a\otimes_{\Gamma_a} F$.
We know that if $w\in W_a,\gamma\in\Gamma$ and $f\in F$,
$$
\Phi(w\otimes_a\gamma\otimes f)|T_h=(w\otimes_a\gamma\otimes_\Gamma f)|T_h=
\sum_{a\in A,g\in\Gamma\backslash S/\Gamma} h(g)
((w\otimes_a\gamma)g\otimes_\Gamma fg).
$$
Compute
$(w\otimes_a\gamma)g\in W$ as follows: 
$(w\otimes_a\gamma)g=((w\otimes_a 1)\gamma)g=w\gamma g$.   Write
$a\gamma g =b(a,\gamma,g)\delta(a,\gamma,g)$ with $b(a,\gamma,g)\in A, \delta(a,\gamma,g)\in\Gamma$.  Then $w\in W_a$ implies that 
$w\gamma g\in W_{b(a,\gamma,g)}\delta(a,\gamma,g)$ which we can write as
$$
(w\gamma g\delta(a,\gamma,g)^{-1}\otimes_{b(a,\gamma,g)}1) \delta(a,\gamma,g)
=w\gamma g\delta(a,\gamma,g)^{-1}\otimes_{b(a,\gamma,g)} \delta(a,\gamma,g).
$$
Thus
$$
\Phi(w\otimes_a\gamma\otimes f)|T_h=
\sum_{a,g} h(g) w\gamma g\delta(a,\gamma,g)^{-1}\otimes_{b(a,\gamma,g)} \delta(a,\gamma,g)\otimes_\Gamma fg.
$$
If $z=\sum_a w^a\otimes_a\gamma^a\otimes f^a$, 
we obtain:
$$
\Phi(z)|T_h = \sum_{g\in\Gamma\backslash S/\Gamma}\sum_a 
h(g) w^a\gamma^a g\delta(a,\gamma^a,g)^{-1}\otimes_{b(a,\gamma^a,g)} \delta(a,\gamma^a,g)\otimes_\Gamma f^ag.
$$
We interchange the order of summation, and then given a fixed $a$, we can choose the representatives $g$ of the double cosets in a way that depends on $a$, without changing the value of the sum.  If $ag\in X_b$ then we choose $g$ so that $ag = b$.  Call such a choice $g_{ab}$.  In other words, we always have
$$
ag_{ab}=b.
$$
Since we want to compute $T_h$ on $\oplus_a W_a\otimes_{\Gamma_a} F$ via $\lambda$, without loss of generality we may take $\gamma^a=1$ for all $a$.  
Then $b(a,1,g_{ab})=b$ and $\delta(a,1,g_{ab})=1$.
We obtain
$$
\Phi\left(\sum_a w^a\otimes_a 1 \otimes f^a\right)|T_h
=\Phi\left(\sum_a\sum_b\sum_{g_{ab}\in\Gamma\backslash S/\Gamma:ag=b}
h(g_{ab})w^a g_{ab}\otimes_b 1\otimes_\Gamma f^ag_{ab}\right).
$$
Via the isomorphism $\lambda$, this corresponds to
$$
\left(\sum_a w^a \otimes_{\Gamma_a} f^a\right)|T_h=
\sum_a\sum_b\sum_{g\in\Gamma\backslash S/\Gamma:ag=b}
h(g)w^ag\otimes_{\Gamma_b}  f^ag.
$$
Call this Formula (1).

Define $h_{ab}(g) = h(g)$ if $ag=b$ and $=0$ otherwise.  Clearly, $h_{ab}\in\FF(\Gamma_a\backslash S/\Gamma_b)$.  Corresponding to $h_{ab}$ is a Hecke operator  $T_{ab}$.  It maps $\Gamma_a$-homology to $\Gamma_b$-homology.

\begin{theorem}\label{hecke-equ}
 Let $W$ be an $S$-sheaf on the $G$-set $X$ and $F$ a right $S$-module.  Let $A$ be a set of representatives of the $\Gamma$-orbits of $X$.  Let
$$
\lambda: W\otimes_{\Gamma} F \to \oplus_a W_a\otimes_{\Gamma_a} F
$$
be the natural isomorphism described above. Then $T_h$ on the left is equivariant to the matrix 
$T_{ab}$ on the right. 

If $F$ is a resolution of $Z$ by projective $\Z[\Gamma]$-modules which are also $S$-modules, then $\lambda$ induces the isomorphism on homology given by Shapiro's lemma and we have
$$
H_q(\Gamma, W) \approx 
\oplus_{a\in A} H_q({\Gamma_a}, W_a), 
$$
and again $T_h$ on the left is equivariant to the matrix 
$T_{ab}$ on the right.
\end{theorem}
\begin{proof}
Without loss of generality, $h$ is the characteristic function of 
$\Gamma s\Gamma=\coprod s_\alpha\Gamma$.  We fix $a,b$.  
If $as_\alpha\in X_b$, we choose $s_\alpha=g_{ab}$.

Then the term in Formula (1) corresponding to $a,b$ is
$$
\Theta_{ab}:=
\sum_{s_\alpha:as_\alpha\in X_b}
h(s_\alpha)w^as_\alpha\otimes_{\Gamma_b}  f^as_\alpha.
$$ 
We must show that
$$
\Theta_{ab}=(w^a \otimes_{\Gamma_a} f^a)|T_{ab}.
$$
To compute $|T_{ab}$, write $\Gamma_a s \Gamma_b =\coprod t\Gamma_b$.  Then
$$
(w^a \otimes_{\Gamma_a} f^a)|T_{ab} = \sum_t h_{ab}(t)w^at\otimes_{\Gamma_b} f^at.
$$
Now $h_{ab}(t)=0$ unless $at=b$.  So if 
$h_{ab}(t)\ne 0$, since $t\in\Gamma s\Gamma$, we have $t=s_\alpha\gamma$ for some $\alpha,\gamma$, and 
$b=at=as_\alpha\gamma=b\gamma$ (since $s_\alpha=g_{ab}$.)  It follows that $\gamma\in\Gamma_b$.  

In other words, $h_{ab}(t)=0$  unless $t\in s_\alpha\Gamma_b$ for some $\alpha$ for which $as_\alpha=b$, and in this case $h_{ab}(t)=h(t)$. Therefore 
$$
(w^a \otimes_{\Gamma_a} f^a)|T_{ab} = 
\sum_{t\in s_\alpha\Gamma_b:as_\alpha\in X_b} h(t)w^at\otimes_{\Gamma_b} f^at =\sum_{s_\alpha:as_\alpha\in X_b}h(s_\alpha)w^as_\alpha\otimes_{\Gamma_b} f^as_\alpha 
$$
which equals $\Theta_{ab}$.

If $F$ is a complex, the action of the Hecke operators on 
$W\otimes_{\Gamma} F$ and $\oplus_a W_a\otimes_{\Gamma_a} F$
commutes with the boundary maps in $F$.
Therefore $T_h$ on 
$H_q(\Gamma, W)$ 
and $T_{ab}$ on 
$\oplus_{a\in A} H_q({\Gamma_a}, W_a)$
are equivariant with respect to $\lambda$.
\end{proof}

\section{Preparing to compute Hecke Operators in $\GL_3$}\label{preparing}

In this section we determine the $g_{ab}$'s that we need to study reducible 3-dimensional Galois representations.

Let $P_0$ be the stabilizer of the line spanned by $(1,0,\ldots,0)$ in affine $n$-space, on which $\GL(n)$ acts on the right.  Note that the elements of $P_0$ are characterized by the fact that all entries in the top row except for the first are zero.

We set 
$$U_0=\begin{pmatrix} 1&0&\cdots&0\cr
*&1&\cdots&0\cr
&&\cdots&\cr
*&0&\cdots&1\cr\end{pmatrix},L_0^1=\begin{pmatrix}*&0&\cdots&0\cr
0&1&\cdots&0\cr
\cdots\cr
0&0&\cdots&1\end{pmatrix},L_0^2=\begin{pmatrix}1&0&\cdots&0\cr
0&*&\cdots&*\cr
\cdots\cr
0&*&\cdots&*\cr
\end{pmatrix}.$$
For $g\in P_0$, we define $\psi_0^1(g)\in\GL(1)$ and $\psi_0^2(g)\in\GL(n-1)$  by
$$g=\begin{pmatrix}\psi_0^1(g)&0\cr*&\psi_0^2(g)\end{pmatrix}.$$
We set $$g_x=\begin{pmatrix}1&x&0\cr0&1&0\cr 0&0&I_{n-2}\end{pmatrix},$$
and we define
$$P_x=g_x^{-1}P_0g_x,\ U_x=g_x^{-1}U_0g_x,L_x^1=g_x^{-1}L_0^1g_x,L_x^2=g_x^{-1}L_0^2g_x.$$
For $s\in P_x$, we set $\psi_x^i(s)=\psi_0^i(g_xsg_x^{-1})$.

We have the following theorem (the steps of the proof are identical to those in \cite[Theorem 7]{A}, after transposing and replacing $d$ by $-d$):
\begin{theorem}\label{psid} Let $d$ be a positive divisor of $N$, and assume $\gcd(d,N/d)=1$.  Then
\begin{enumerate} 
\item $U_dL_d^1\cap \Gamma_0(n,N)=U_d\cap\Gamma_0(n,N)$.
\item If $s\in P_d\cap S_0(n,N)^\pm$, then $\psi_d^1(s)\equiv s_{11}\pmod d$ and $\psi_d^2(s)_{11}\equiv s_{11}\pmod{N/d}$.
\item  $\psi_d^2(P_d\cap S_0(n,N)^\pm)\subset S_0(n-1,N/d)^\pm$.
\item $\psi_d^2$ induces an exact sequence
$$1\to U_d\cap \Gamma_0(n,N)\to P_d\cap \Gamma_0(n,N)\too{\psi_d^2}\Gamma_0(n-1,N/d)^\pm\to 1.$$
\end{enumerate}
\end{theorem}

In order to compute Hecke operators with respect to the Hecke pair $(S_0(3,N),\Gamma_0(3,N))$, we use the coset representatives described in the next theorem.

\begin{theorem}\label{cosets} We have the following coset decompositions of double cosets $$\Gamma_0(3,N)s\Gamma_0(3,N)=\bigcup_{g\in C}g\Gamma_0(3,N).$$

\begin{enumerate} 
\item For $s=\diag(1,1,\ell)$, with $\ell$ prime and $(\ell,N)=1$, $$C=\left\{\begin{pmatrix}1&0&0\cr 0&1&0\cr b&c&\ell\end{pmatrix},\begin{pmatrix}1&0&0\cr a&\ell&0\cr 0&0&1\end{pmatrix},\begin{pmatrix}\ell&0&0\cr 0&1&0\cr 0&0&1\end{pmatrix}:0\leq a,b,c\leq \ell-1\right\}$$
\item For $s=\diag(1,\ell,\ell)$, with $\ell$ prime and $(\ell,N)=1$,  $$C=\left\{\begin{pmatrix}1&0&0\cr a&\ell&0\cr b&0&\ell\end{pmatrix},\begin{pmatrix}\ell&0&0\cr 0&1&0\cr 0&c&1\end{pmatrix},\begin{pmatrix}\ell&0&0\cr 0&\ell&0\cr 0&0&1\end{pmatrix}:0\leq a,b,c\leq \ell-1\right\}$$
\end{enumerate}

\end{theorem}
\begin{proof} One checks that the given elements are all in the double coset, and that none are in the same coset of $\Gamma_0(3,N)$.  Because the cardinality of $C$ is equal to the number of cosets of $\Gamma_0(3,N)$ in the double coset, they must form a complete set of coset representatives.
\end{proof}

The next theorem is adapted to the following situation:  Using the notations of Section~\ref{g-sets}, let $G=\GL(3,\Q)$, $X=\PP^2(\Q)$, $\Gamma=\Gamma_0(3,N)$.  
We assume $N$ is squarefree, in which case, as proved in \cite{A},
the $\Gamma$-orbits of $X$ may be represented by the set 
$$
A=\{(1:d:0)\ | \ d>0,\ \ d|N \}.
$$
Also, any $s$ as in Theorem~\ref{cosets} takes each $\Gamma$-orbit to itself.  For each $s$ the following theorem gives a $\gamma$ such that $(1:d:0)s\gamma=(1:d:0)$.  Thus $g_{ab}=0$ unless $a=b$ and then $g_{aa}=s\gamma$, where $a=(1:d:0)$.  The theorem also gives the values of $\psi_0^i,i=1,2$ which we will need to compute the action  of $\Gamma_a$ on the $S$-sheaves we deal with in Sections~\ref{spec} and~\ref{redgal}.

\begin{theorem}\label{calcx} Let $\ell$ be a prime not dividing $N$, let $d$ be a divisor of $N$ with $(d,N/d)=1$ and let $s$ be a matrix of the form
$$s=\begin{pmatrix}\ell_1&0&0\cr a&\ell_2&0\cr b&c&\ell_3\end{pmatrix}$$
that is in one of the sets $C$ in Theorem~\ref{cosets}.  Then there exists a $\gamma\in\Gamma_0(3,N)$ of the form
$$\gamma=\begin{pmatrix}A&B&0\cr C&D&0\cr 0&0&1\end{pmatrix}$$
such that $s\gamma\in P_d=g_d^{-1}P_0g_d$, and for $x=g_ds\gamma g_d^{-1}\in P_0$ we have
\begin{enumerate}
\item If $\ell_1=\ell_2$ and $a=0$, then $x_{11}=\ell_1$ and
$$\psi_0^2(x)=\begin{pmatrix}\ell_2&0\cr c-bd&\ell_3\end{pmatrix},$$
\item If $\ell_1=\ell$, $\ell_2=1$ and $a=0$, then $x_{11}=1$ and
$$\psi_0^2(x)=\begin{pmatrix}\ell&0\cr -bd+c\ell&\ell_3\end{pmatrix}.$$
\item If $\ell_1=1$, $\ell_2=\ell$, and $\ell\nmid ad+1$, then $x_{11}=1$ and
$$\psi_0^2(x)=\begin{pmatrix}\ell&0\cr -b\ell d+c(ad+1)&\ell_3\end{pmatrix}.$$
\item If $\ell_1=1$, $\ell_2=\ell$, and $\ell| ad+1$, then $x_{11}=\ell$ and
$$\psi_0^2(x)=\begin{pmatrix}1&0\cr -bd+c\frac{ad+1}{\ell}&\ell_3\end{pmatrix}.$$
\end{enumerate}
\end{theorem}

We give the proof of this theorem in the appendix.

\section{Irreducible representations}\label{irreducible}
Let $B_m$ be the Borel subgroup of $\GL(m)$ consisting of upper triangular matrices, and let $T_m$ be the maximal torus of diagonal matrices.  An algebraic weight with respect to the pair $(B_m,T_m)$ is an $m$-tuple of integers $(a_1,\ldots,a_m)$ which represents the map $\diag(t_1,\ldots t_m)\mapsto t_1^{a_1}\cdots t_m^{a_m}$.  This weight is dominant if $a_1\geq a_2\geq\cdots\geq a_m$.

A dominant weight is said to be  { $p$-restricted} if $0\leq a_i-a_{i+1}\leq p-1$ for $1\leq i< m$ and $0\leq a_m\leq p-2$.  For any $p$-restricted weight there exists a unique (up to isomorphism) irreducible right $\F_p[\GL(m,\bar\F_p)]$-module $F(a_1,\ldots,a_m)$ with highest weight $(a_1,\ldots,a_m)$ which remains irreducible when restricted to $\F_p[\GL(m,\Z/p\Z)]$, and all such irreducible  $\F_p[\GL(m,\Z/p\Z)]$-modules occur this way \cite[p. 412]{Doty-Walker}.  These modules are, in fact, absolutely irreducible \cite[Corollary II.2.9]{J2}.  We will relax the condition that $0\leq a_m\leq p-2$, and allow $a_m$ to be arbitrary,  stipulating that we can adjust the entire $m$-tuple by adding the same multiple of $p-1$ to each entry without changing the corresponding module.  This is equivalent to tensoring with the $(p-1)$ power of the determinant, which changes the module over $\F_p[\GL(m,\bar\F_p)]$, but not over $\F_p[\GL(m,\Z/p\Z)]$.

In a similar way, all the irreducible $\F_p[\GL(1,\Z/p\Z)\times \GL(m-1,\Z/p\Z)]$-modules are classified by $m$-tuples $(a_1,\ldots,a_m)$ such that $(a_2,\ldots,a_m)$ is $p$-restricted for $\GL(m-1)$ and $a_1$ is considered modulo $p-1$.  We will denote the $\F_p[\GL(1,\Z/p\Z)\times \GL(m-1,\Z/p\Z)]$-module corresponding to $(a_1,\ldots,a_m)$ by $M(a_1;a_2,\ldots,a_m)$.  We note that this module is just $F(a_2,\ldots,a_m)$ as an $\F_p[\GL(1,\Z/p\Z)\times \GL(m-1,\Z/p\Z)]$-module, with $\GL(1,\Z/p\Z)$ acting as scalars via the $a_1$-power map.

Note that we will also consider any $\GL(m,\Z/p\Z)$-module as a $\GL(m,\Z_p)$-module via reduction modulo $p$.  In this way, we make any $\GL(m,\Z/p\Z)$-module into an $S_0(m,N)$-module.  Note that if $N$ and $p$ are relatively prime, then the image of $S_0(m,N)$ under reduction modulo $p$ is all of $\GL(m,\Z/p\Z)$.

Given a  $\GL(m,\Z_p)$-module $E$, we will denote the action of $s\in\GL(m,\Z_p)$ by $e|s$.  We denote by $E^x$ the module with the same underlying abelian group as $E$, but with $s\in\GL(m,\Z_p)$ acting via $e|^xs=e|g_xsg_x^{-1}$.  We note that if $E=F(a_1,\ldots,a_m)$ is an irreducible module, then $E^x$ is also isomorphic to $F(a_1,\ldots,a_m)$, since it is irreducible of the same dimension, with the same highest weight viewed as an $\F_p[\GL(m,\bar\F_p)]$-module (since conjugation by $g_x$ does not introduce any new torus characters).  For a Dirichlet character $\chi$ of conductor $N$,  we denote by $E^x_\chi$ the $S_0(n,N)$ module with the same underlying abelian group as $E$ but with the action of $s$ given by $e|^x_\chi s=\chi(s)e|g_xsg_x^{-1}$, where $\chi(s)$ is defined as $\chi(s_{11})$.

\begin{theorem}\label{U-inv}
Let $n\geq 2$, let $(a_1,\ldots,a_n)$ be a p-restricted weight, and set $F=F(a_1,\ldots,a_n)$.  Let $\chi$ be a Dirichlet character of conductor $N$, which factors as $\chi_0\chi_1$, where $\chi_0$ has conductor $d$, and $\chi_1$ has conductor $N/d$ with $(d,N/d)=1$. Then 

(1) The module of invariants $F^{U_0(\Z/p)}$ is isomorphic to $M(a_n;a_1,\ldots,a_{n-1})$.

(2) The module $F^{U_d\cap S_0(n,N)}$ considered as a $P_d\cap S_0(n,N)$-module is isomorphic to $(F^{U_0})^d$.

(3) The action of $P_d\cap S_0(n,N)$ on the module $F_\chi^{U_d\cap S_0(n,N)}$ is given by
$$e|s=\chi_0(\psi_d^1(s))(\psi_d^1(s))^{a_n}\chi_1(\psi_d^2(s))e|(\psi_d^2(s)),$$
where the vertical bar on the right denotes the action of $\GL(n-1,\Z_p)$ on $F(a_1,\ldots,a_{n-1})$.  

\end{theorem}
\begin{proof} 

(1) Set $U=U_0(\Z/p)$, and $L=P_0/U=L_0^1(\Z/p)\times L_0^2(\Z/p)$.  Let $A$ be the outer automorphism of $\GL(n)$ given by $A(g)={}^tg^{-1}$.  The contragredient of $F$, or $F^\vee$ is a $\GL(n,\Z_p)$-module with the same underlying abelian group as $F$, but with the action given by $f|^\vee g=f|A(g)$.  We note that as $\GL(n,\Z_p)$-modules, the contragredient and the dual, $\Hom(F,\bar\F_p)$ are isomorphic.

As $L$-modules, we have that $F^U\cong ((F^\vee)^{A(U)})^\vee$.  Then by \cite[Proposition 5.10]{Humphreys} (see also \cite[Lemma 2.5]{H1}), we see that
$$F^U\cong (F(-a_n,\ldots,-a_1)^{A(U)})^\vee\cong M(-a_n;-a_{n-1}\ldots,-a_1)^\vee\cong M(a_n;a_1,\ldots,a_{n-1}).$$

(2) Since $N$ and $p$ are relatively prime, $U_d\cap S_0(n,N)$ and $U_d$ have the same image in $\GL(n,\Z/p\Z)$, so we need only consider $F^{U_d}$.  We see that $F^{U_d}=((F^{-d})^{U_0})^d$.  However, as described above, $F^{-d}\isom F\isom F(a_1,\ldots,a_n)$.  Hence, by part (1), we find that $F^{U_d}\isom M(a_n;a_1,\ldots,a_{n-1})^d$.

(3) We identify $F^{U_d\cap S_0(n,N)}$ with $M(a_n;a_1,\ldots,a_{n-1})^d$.  Then, for $e\in F^{U}$, we have
\begin{align*}
e|^d_\chi s&=\chi(s)e|g_dsg_d^{-1}\cr&=\chi(s)e|(\psi_d^1(s)\times\psi_d^2(s))\cr&=\chi(s)(\psi_d^1(s))^{a_n}e|\psi_d^2(s)\cr&=\chi_0(\psi_d^1(s))\chi_1(\psi_d^2(s))(\psi_d^1(s))^{a_n}e|\psi_d^2(s)\end{align*}
since $e\in F^U\cong M(a_n;a_1,\ldots,a_{n-1})$ as a $\GL(1)\times\GL(n-1)$-module and 
$$\chi(s)=\chi(s_{11})=\chi_0(s_{11})\chi_1(s_{11})=\chi_0(\psi_d^1(s))\chi_1(\psi_d^2(s))$$ by Theorem~\ref{psid}(2).

\end{proof}

\section{Points in general position}
Let $K$ be an infinite field, $V$ an $n$-dimensional $K$-vector space, and denote the projective space by $\mathbb P=\mathbb P(V)$.  Fix a basis of $V$.  Given points $a_1,\ldots,a_r\in \mathbb P$, we say that the points are in general position if any subset of them spans a linear space of maximal possible dimension.  We define a simplicial complex $Y^g=Y^g(K)$ as follows.  The vertices of $Y^g$ are points in $\mathbb P$, and are acted upon by $\GL_n(K)$ on the right.  The $p$-simplices of $Y^g$ are spanned by $(p+1)$-tuples of vertices that are in general position. We let $X^g$ denote the chain complex of oriented chains on $Y^g$.  The augmentation is the map $\varepsilon:X^0\to\Z$ that sends each vertex to $1$. If $x_{0},\ldots,x_{p}$ are vertices spanning a $p$-simplex $\Delta$ in $X^g$, we denote by $\vec x=(x_{0},\ldots,x_{p})$ the chain supported on $\Delta$ with coefficient $1$.  It is antisymmetric in the arguments.  Then $X^g_p$ is generated over $\Z$ by these basic chains $\vec x$.

\begin{theorem} If $K$ is infinite, then $X^g$ is an acyclic resolution of $\Z$ by $\GL(n,K)$-modules.
\end{theorem}
\begin{proof} Let $\vec x_i=(x_{i0},\ldots,x_{ip})$ be a basic chain supported on a $p$-simplex for each $i$, and suppose that $z=\sum_{i}c_i\vec x_i$ is a cycle, i.e. a $p$-chain that is taken to zero by the boundary map.

Since $K$ is infinite, we may choose a point $y\in \mathbb P$ such that for each $i$, the set $\{x_{i0},x_{i1},\ldots,x_{ip},y\}$ is in general position.  Let $y_i=( x_{i0},x_{i1},\ldots,x_{ip},y)$.

Then 
\begin{align*}
\partial\sum_ic_iy_i&=\sum_ic_i\partial y_i\cr
&=\sum_i\sum_{j=0}^p(-1)^{j+1}c_i(x_{i0},\ldots,\widehat x_{ij},\ldots,x_{ip},y)+(-1)^{p+2}\sum_ic_i(x_1,\ldots,x_p,\widehat y)\cr
&=0\pm z,\end{align*}
where the double sum is 0 since $z$ is a cycle.  Adjusting the sign, we see that every cycle is a boundary. 
\end{proof}

\section{A spliced sharbly complex}
In this section, $K$ is an infinite field.
We are going to splice the complex $X^g$ with the sharbly complex.  Recall the Steinberg module $St_n$ and the sharbly complex $Sh^n$ for an $n$-dimensional $K$-vector space $V$ as described for example in~\cite{AGMV}. If $v_1,\dots,v_{n+k}$ are vectors in $V$, denote by $[v_1,\dots,v_{n+k}]$ the basic $k$-sharbly.  It is antisymmetric in the arguments, it doesn't change if any argument is multiplied by a non-zero element of $K$, and it vanishes if the arguments do not span $V$ over $K$.  An element $g\in\GL(n,K)$ acts on a basic $k$-sharbly by $[v_1,\ldots,v_{n+k}]g=[v_1g,\ldots,v_{n+k}g]$.  The $\Z$-span of the basic $k$-sharblies, subject to these relations, is by definition $Sh^n_k(V)$.  The boundary map
$Sh^n_{k+1}(V) \to Sh^n_k(V)$ is given by $[v_1,\dots,v_{n+k+1}] \mapsto 
\sum_i (-1)^i[v_1,\dots,\hat v_i, \dots v_{n+k+1}]$.

The module $St_n(V)$, which is free as a $\mathbb Z$-module, is isomorphic as a $\GL(n,K)$-module to the cokernel of 
the boundary map $Sh^n_1(V)\to Sh^n_0(V)$.
For $[v_1,\ldots,v_n]\in Sh^n_0(V)$, we denote the image of $[v_1,\ldots,v_n]$ in the cokernel $St_n(V)$ by $\{v_1,\ldots,v_n\}$.

If $K=\Q$,
Borel-Serre duality, as improved by Brown \cite[X.3.6]{B}, gives us Hecke-equivariant isomorphisms $H_i(\Gamma,St_n\otimes M) \approx H^{n(n-1)/2-i}(\Gamma, M)$ for any subgroup of finite index $\Gamma \subset \GL(n,\Z)$ and any $S$-module $M$ on which $(n+1)!$ acts invertibly.

We know that 
$$
\dots \to Sh^n_i(V) \to Sh^n_{i-1}(V) \to Sh^n_0(V) \to St_n(V) \to 0
$$
is an exact sequence of $\GL(n,K)$-modules.

Sending 
$(v_1,\dots,v_{n+k})$ to $[v_1,\dots,v_{n+k}]$ defines an injective map of   $\GL(n,K)$-modules 
$\iota: X^g_{n+k-1}\to Sh_k$.  Note that $\iota$ induces an isomorphism 
$X^g_{n-1} \approx Sh_0$.

We now set $n=3$ and $V=K^3$, dropping the $n$ and the $V$ from the notation $Sh^n_i(V)$. 
Define a new complex $X$ of $\GL(3,K)$-modules as follows.  For $i\ge 2$, $X_i=Sh_{i-2}$ and the boundary map $X_{i+1}\to X_i$ is the same as in the sharbly complex.  We define $X_0=X^g_0$ with the same augmentation map $\varepsilon:X_0\to\Z$.  It remains to define $X_1$ and the boundary maps $X_2\to X_1\to X_0$.

Let $\mathbb P^\ast$ denote the set of planes in $K^3$.  We set 
$X_1=\bigoplus_{H\in \mathbb P^\ast} St_2(H)$.  An element $g\in \GL(3,K)$ acts on $\{a,b\}\in X_1$ by sending it to $\{ag,bg\}$.  

Note that $X_2$ is generated freely over $\Z$ by ``generic sharblies", i.e. $[a,b,c]$ such that the determinant of the matrix with rows $a,b,c$ is nonzero.  Define the boundary map $\partial_2:X_2\to X_1$ by $[a,b,c]\mapsto \{a,b\}+\{b,c\}+\{c,a\}$ for any generic $[a,b,c]$.  It is well-defined and is $\GL(3,K)$-equivariant.
We define the boundary map
$\partial_1:X_1\to X_0$ by $\{a,b\}\mapsto (b)-(a)$.  

\begin{lemma}
Let $R$ be the submodule of $X^g_1$ generated by
$$
\{(a,b)+(b,c)+(c,a)\ | \ H \in \mathbb P^\ast, \ a,b,c\in H, \ (a,b,c)\in X^g_2(H)\}.
$$
For $H\in \mathbb P^\ast$ and $a,b\in H$, let $\psi(a,b)=\{a,b\}\in St_2(H)$.
Then

(a) The map $\psi: X^g_1\to X_1$ induces an isomorphism of $\GL(3,K)$-modules
$\phi:X^g_1/R\to X_1$.

(b) 
The boundary map in the generic complex $X^g_1\to X^g_0=X_0$ induces the 
$\GL(3,K)$-equivariant boundary map
$\partial_1:X_1\to X_0$ after identifying $X_1$ with $X^g_1/R$ via $\phi$. 

\end{lemma}
\begin{proof}  
(a) The map $\psi$ is clearly $\GL(3,K)$-equivariant and surjective.  If $t\in X^g_1$, we can write $t=\sum_{H\in \mathbb P^\ast} t_H$, where $t_H$ is supported on symbols $(u,v)$ with $u,v\in H$.  Then $\psi(t)=0$ if and only if $\psi_H(t)=0$ for all $H$. 
For each $H$, $St_2(H)$ is the cokernel of the boundary map $\beta:Sh_1(H)\to Sh_0(H)$.  The image of $\beta$ is generated by  the basic relations $\{u,v\}+\{v,w\}+\{w,u\}$ where $(u,v,w)$ runs over triples in $H$ such that $u,v,w$ generate pairwise distinct lines.
Therefore the kernel of $\psi$ is exactly $R$. 

(b)
The boundary map $\partial_1:X_1\to X_0$ in $X^g$ sends $(a,b)$ to $(b)-(a)$.  This contains $R$ in its kernel, and induces a map $X_1/R\to X_0$ which clearly becomes $\partial_1$ after the identification via $\phi$.  It is obviously $\GL(3,K)$-equivariant.
\end{proof}

\begin{lemma} The sequence of $\Z[\GL(3,K)]$-modules
$$
\dots \to X_i \to X_{i-1} \to \dots \to X_2 \to X_1 \to X_0 \to \Z \to 0
$$
is exact.
\end{lemma}

\begin{proof}
Since the sharbly sequence is exact and the augmentation map is surjective, the only nodes which need checking are those at $X_i$, $i=0,1,2$.  Denote the boundary map from $X_i\to X_{i-1}$ by $\partial_i$.  Denote the boundary maps in $X^g$ by $\partial^g$.  We identify $X^g_{k+2}$ with its image in $Sh_k$ via $\iota$, under which identification (1) $\partial^g_{k+2}=\partial_k|X^g_{k+2}$ and (2) $X^g_2$ and $Sh_0$ are isomorphic.

Node at $X_0$: Clearly $\varepsilon\circ\partial_1=0$.  Let $x\in X_0$.  If $\varepsilon(x)=0$, there exists $y\in X^g_1$ with $\partial^g_1(y)=x$ because $X^g$ is exact.  Then $\partial_1(\psi y)=x$.

Node at $X_1$: Clearly $\partial_1\circ\partial_2=0$.  Let $y\in X_1$ such that  $\partial_1(y)=0$.  Choose $y'\in X^g_1$ such that $y=\psi y'$.  Then 
$\partial^g_1(y')=0$.
Hence there exists $z\in X^g_2=Sh_0 = X_2$ with 
$\partial^g_2(z)= y'$ because $X^g$ is exact.  Then $\partial_2(z)=y$.

Node at $X_2$: Clearly $\partial_2\circ\partial_3=0$.
Now suppose $x\in X_2$ such that $\partial_2(x)=0$.  Then 
$\partial^g_2(x)\in R$.
We will show (*) for every $r\in R$ there exists $\widetilde r\in X_3=Sh_1$ such that $r=\partial^g_2\partial_3(\widetilde r)\in X^g_1$. Then 
$\partial^g_2(x - \partial_3(\widetilde{\partial^g_2(x)})=0$.  
Therefore there exists $t\in X^g_3\subset X_3$
such that $\partial_3(t)=\partial^g_3(t) = x - \partial_3(\widetilde{\partial^g_2(x)})$.  Then $\partial_3(t+\widetilde{\partial^g_2(x)})=x$.

Proof of (*): Let 
$H \in \mathbb P^\ast, \ a,b,c\in H, (a,b,c)\in X^g_2(H)$.  (The last condition just means that $a,b,c$ are pairwise distinct.)  
It suffices to let
$r=(a,b)+(b,c)+(c,a)$ and find $\widetilde r$.   Pick $v\not\in H$.  Set $\widetilde r=[a,b,c,v]$.  Then $\partial_3(\widetilde r)=[b,c,v]-[a,c,v]+[a,b,v]$ since $[a,b,c]=0$ in $Sh_0$.  Hence $\partial^g_2(\partial_3(\widetilde r))=r$.
\end{proof}

A straightforward attempt to generalize this lemma for $n>3$ doesn't work.    The reason the lemma works for $n=3$ is that any tuple of pairwise distinct lines in a plane is generic.

We now specialize to the case in which $K=\Q$, and we let $\Gamma$ be a subgroup of finite index in $\SL(3,\Z)$.  If $M$ is any $S$-module, $X_1\otimes M = \bigoplus_{H\in \mathbb P^\ast} St_2(H)\otimes M$. If $H$ is a fixed plane, let $P_H$ be its stabilizer in $\GL(3)$ and $U_H$ the unipotent radical of $P_H$.  Since $\SL(3,\Z)$ acts transitively on $\mathbb P^\ast$, 
$\Gamma$ has a finite number of orbits in $\mathbb P^\ast$, represented by the planes in a finite set, say $\bH(\Gamma)$.
It follows from \cite[Corollary III.5.4]{B} that $X_1 \otimes M$ is isomorphic as $S$-module to a finite sum of induced modules 
$\bigoplus_{H\in\bH(\Gamma)}\Ind_{P_H\cap\Gamma}^{\Gamma} St_2(H)\otimes M$, where $U_H\cap\Gamma$ acts trivially on $St_2(H)$.

\section{The spectral sequence for $\GL_3$}\label{spec}

Now set $K=\Q$.
Let $F$ be a resolution of $\Z$ by $\Z[S]$-modules which are free as $\Z[\Gamma]$-modules.  
Let $W=X\otimes M$ with the diagonal $S$-action and
form the double complex $W\otimes_\Gamma F$.  To compute the homology, we use the spectral sequence from \cite[VII.5.(5.3)]{B}.  
All the differentials on the $E^i$ page, $i\ge 1$ are Hecke-equivariant because they are all induced by the differential on the double complex, which is an $S$-module.  From now on, we assume that $6$ is invertible on $M$.

Since $X$ is a resolution of $\Z$ by free $\Z$-modules, we see that $W$ is a resolution of $M$.  Therefore there is a weak equivalence from $W$ to the chain complex consisting of $M$ concentrated in dimension 0, so that by \cite[VII.5.2]{B}, $H_*(\Gamma,W)\approx H_*(\Gamma,M)$.

Hence, in the spectral sequence, we have
$$
E_{jq}^1=H_q(\Gamma, X_j\otimes M)\implies H_{j+q}(\Gamma,M).
$$
For each $j\ne 1$, the $j$-chains $X_j$ are isomorphic to a direct sum of induced representations 
$$
X_j\approx \bigoplus_{\sigma\in C_j} \Ind_{\Gamma_\sigma}^\Gamma M_\sigma
$$
where $C_0$ is the set of vertices in $Y^g_0$; if $j\ge 2$, $C_j$ is the set of basic sharblies $[v_1,\dots,v_{j+1}]$ with $v_1,\dots,v_{j+1}$ vectors in $\Q^3$ that span $\Q^3$; $\Gamma_\sigma$ is the stabilizer of $\sigma$, $\varepsilon_\sigma$ is the orientation character recording how $\Gamma_\sigma$ acts on $\sigma$; and $M_\sigma=M\otimes\varepsilon_\sigma$. 

As in \cite[VII.7]{B} we have from Shapiro's lemma:
$$
E_{jq}^1=\bigoplus_{\sigma\in C_j}H_q(\Gamma_\sigma,M_\sigma)
$$
if $j\ne 1$.
We see easily that $H_q(\Gamma_\sigma,M_\sigma)=0$ in the following cases:
\begin{enumerate}
\item For $q>3$ and $\sigma\in C_0$, since the virtual cohomological dimension of the stabilizer $P$ of a $0$-cycle is $3$ and any torsion element of $P$ has order invertible on $M$.
\item For $q\geq 1$ and $\sigma\in C_j$ with $j>1$, since the stabilizer of a basic sharbly is a finite subgroup of $\GL(3,\Z)$, hence of order invertible in $M$ (see \cite[Corollary III.10.2]{B}).
\end{enumerate}

The column with $j=1$ will be treated later.

We then have the $E^1$ page of our spectral sequence as follows, where 
$H_i(\sigma):=H_i(\Gamma_\sigma,M_\sigma)$.

\bigskip
\noindent
\begin{tabular}{ccccc}
$\displaystyle{\bigoplus_{\sigma\in X_0/\Gamma}H_3(\sigma)}$&
$\displaystyle{H_3(\Gamma, X_1\otimes M)}$&0&0&0\cr\ \cr
$\displaystyle{\bigoplus_{\sigma\in X_0/\Gamma}H_2(\sigma)}$&
$\displaystyle{H_2(\Gamma, X_1\otimes M)}$&0&0&0\cr\ \cr
$\displaystyle{\bigoplus_{\sigma\in X_0/\Gamma}H_1(\sigma)}$&
$\displaystyle{H_1(\Gamma, X_1\otimes M)}$&0&0&0\cr\ \cr
$\displaystyle{\bigoplus_{\sigma\in X_0/\Gamma}H_0(\sigma)}$&
$\displaystyle{H_0(\Gamma, X_1\otimes M)}$&
$\displaystyle{H_0(\Gamma,Sh_0\otimes M)}$&$\displaystyle{H_0(\Gamma,Sh_1\otimes M)}$&
$\displaystyle{H_0(\Gamma,Sh_2\otimes M)}$\cr
\end{tabular}

It will follow from the following lemma that any Hecke eigenclass in $E_{03}^1$ that has an attached Galois representation which is the sum of an irreducible two-dimensional representation and a character cannot be hit by any differential in the spectral sequence and hence survives to $E^\infty$.  (This uses the fact that the Hecke eigenvalues determine the characteristic polynomials of Frobenius and hence the attached Galois representation up to semi-simplification.)

\begin{theorem}\label{d}
 Assume that $M$ is an admissible $\F_p[S]$-module with $p>3$.  Then the $E_{40}^2$ and $E_{13}^1$ terms of the spectral sequence are finite dimensional $\F_p$-vector spaces and all systems of Hecke eigenvalues appearing in them are attached to Galois representations that are sums of characters.   
\end{theorem}

\begin{proof}
Since the sharbly complex is a resolution of $St:=St_3(\Q^3)$, and $6$ acts invertibly on $M$, $E^2_{40}\approx H_2(\Gamma,St\otimes M)$, cf. Corollary 8 in~\cite{AGMV}.  
Borel-Serre duality then gives an isomorphism of Hecke-modules
$E^2_{40}\cong H^1(\Gamma,M)$, which is a finite dimensional $\F_p$-vector space.  
By \cite[Theorem 4.1.5]{Ash-Duke}, $H^1(\Gamma,M)$ is a sum of generalized Hecke eigenspaces, and any eigenclass appearing in $H^1(\Gamma,M)$ has an attached Galois representation that is a sum of three characters.  

We now consider $E_{13}^1$.
Recall that there is a finite set of planes $\bH(\Gamma)$ such that
$$
X_1\otimes M \approx 
\bigoplus_{H\in\bH(\Gamma)}\Ind_{P_H\cap\Gamma}^{\Gamma} St_2(H)\otimes M
$$
 where $U_H\cap\Gamma$ acts trivially on $St_2(H)$.

We define the $S$-sheaf of $\PP^*$ by $H\mapsto St_2(H)\otimes M$.  We then use Theorem~\ref{hecke-equ}.

By Shapiro's lemma,  
$E_{13}^1\approx \bigoplus_{H\in\bH(\Gamma)} H_3(P_H\cap \Gamma, St_2(H)\otimes M)$.  

Let $\Gamma_H$ denote $(P_H\cap \Gamma)/(U_H\cap \Gamma)$.  It is isomorphic to a congruence subgroup of $\GL(2,\Z)$.  We have the Hochschild-Serre spectral sequence
$$
E^2_{jq}=H_j(\Gamma_H,H_q(U_H\cap\Gamma,St_2(H)\otimes M))\implies 
H_{j+q}(P_H\cap \Gamma, St_2(H)\otimes M).
$$
Because $U_H\cap\Gamma$ acts trivially on $St_2(H)$, we have $H_q(U_H\cap\Gamma,St_2(H)\otimes M) \approx St_2(H)\otimes C(H,q)$ where $C(H,q)=H_q(U_H\cap\Gamma, M)$ is an admissible $\Gamma_H$-module on which $6$ acts invertibly.   Also, the  homological dimension of $U_H\cap \Gamma$ is $2$. Therefore the only nonzero term in the $E^2$ page of the Hochschild-Serre spectral sequence when $j+q=3$ occurs when $j=1,q=2$.  So any packet of Hecke eigenvalues occurring in
$$H_{3}(P_H\cap \Gamma, St_2(H)\otimes M)$$
also occurs in
$$H_1(\Gamma_H, St_2(H)\otimes C(H,2)).
$$
By Borel-Serre duality, this is isomorphic to
$H^0(\Gamma_H,C(H,2))$ and by \cite[Theorem 4.1.4]{Ash-Duke}, this is a sum of generalized Hecke eigenspaces, and any system of Hecke eigenvalues occurring here has as attached Galois representation a sum of two characters.  At this point we also see that $E_{13}^1$ is finite dimensional over $\F_p$.

Now we use Theorem~\ref{hecke-equ} to compute the Hecke operators on $E^1_{13}$, following the same outline as that used below for $E_{03}^1$.   This proves that $E^1_{13}$
is a sum of generalized Hecke eigenspaces, and any system of Hecke eigenvalues occurring here has as attached Galois representation a sum of three characters.
\end{proof}

\section{Reducible Galois representations}\label{redgal}

We continue to assume that 6 is invertible on $M$.  We note that for $\sigma\in C_0$, the orientation character is trivial, so that $M_\sigma=M$. For any $d|N, d>0$, we have taken $(1:d:0)$ (with stabilizer $P_d$) as a representative of its orbit in $C_0$.
Therefore, $E_{03}^1$ contains $H_3(P_d\cap\Gamma_0(3,N),M)$ as a direct summand.
We will find the system of Hecke eigenvalues in which we are interested in this summand.

Let $\rho:G_\Q\to\GL_3(\bar\F_p)$ be a Galois representation that can be written as a direct sum $\rho=\sigma\oplus\psi$, where $\sigma:G_\Q\to\GL_3(\bar\F_p)$ is an odd, irreducible, two-dimensional representation, and $\psi:G_\Q\to\GL_1(\bar\F_p)$ is a character.  Let $N_1$ be the Serre conductor of $\sigma$, and let $d$ be the conductor of $\psi$, and assume that $N_1d$ is squarefree. We will show that $\rho=\sigma\oplus\psi$ is attached to a Hecke eigenclass in $E_{3,0}^2$ with $\Gamma=\Gamma_0(3,N)$, and coefficient module as predicted by Conjecture~\ref{conjecture}.

We set $\tau=\sigma\otimes\omega^{-1}$, so that $\rho=(\tau\otimes\omega)\oplus \psi$.  Assume that the predicted weight of $\tau$ in Serre's conjecture is $F(a,b)$, with $0\leq a-b\leq p-1$ and $0\leq b<p-1$.  Then if $\tau$ is ordinary, we have
$$\tau|_{I_p}\sim\begin{pmatrix}\omega^{a+1}&*\cr0&\omega^b\end{pmatrix}$$
and if $\tau$ is supersingular, we have
$$\tau|_{I_p}\sim\begin{pmatrix}\omega_2^{(a+1)+bp}&0\cr0&\omega_2'^{(a+1)+bp}\end{pmatrix}.$$
Note that $\tau$ has Serre conductor $N_1$ (the same as $\sigma$), and we may factor $\det(\tau)=\omega^{(a+b+1)}\chi_1$, where the conductor of $\chi_1$ divides $N_1$.  We may also factor $\psi=\omega^c\chi_0$, where $\chi_0$ has conductor $d$ and $0\leq c<p-1$.  We will denote by $\lambda_\ell$ the trace of $\tau(\frob_\ell)$.  We note then that the trace of $\rho(\frob_\ell)$ is equal to $\ell\lambda_\ell+\chi_0(\ell)\ell^c$ and its cotrace (the coefficient of $X^2$ in $\det(I-\rho(\frob_\ell))$) is equal to $\ell(\chi_1(\ell)\ell^{a+b+2}+\chi_0(\ell)\ell^c\lambda_\ell)$.  Examining Conjecture~\ref{conjecture}, we see that the predicted weight of $\rho$ is $F(a,b,c)$ (where, if necessary we add $p-1$ to both $a$ and $b$ so that $(a,b,c)$ will be $p$-restricted, and note that by our conventions this change does not change the module $F(a,b)$).  In addition, the nebentype of $\rho$ is $\chi_0\chi_1$, and the level of $\rho$ is $N=N_1d$.

We know by Serre's conjecture, which is now a theorem, that $\tau$ is attached to a Hecke eigenclass in $H_1(\Gamma_0(2,N_1),F(a,b)_{\chi_1})$.  Since $\tau$ is absolutely irreducible, it is attached to a cusp form, and by Eichler-Shimura it is also attached to a Hecke eigenclass $f$ in $$H_1(\Gamma_0(2,N_1)^\pm,F(a,b)_{\chi_1}).$$

Denote by $M$ the module $F(a,b,c)_\chi$. We now consider the Hochschild-Serre spectral sequence for the homology of the exact sequence
$$1\to U_d\cap \Gamma_0(3,N)\to P_d\cap \Gamma_0(3,N)\stackrel{{\psi_d^2}}\to\Gamma_0(2,N/d)^{\pm}\to 1$$
with coefficients in $M$.

This spectral sequence degenerates at $E^2$ because it is only two columns thick, and we have $$E^2_{jq}=H_j(\Gamma_0(2,N/d)^\pm,H_q(U_d\cap\Gamma_0(3,N),M)).$$
The desired Hecke eigenclass stems from $H_1(\Gamma_0(2,N/d)^\pm,H_2(U_d\cap\Gamma_0(3,N),M))$.

Note that since $U_d\cap\Gamma_0(3,N)$ is an abelian group of rank 2, $$H_2(U_d\cap\Gamma_0(3,N),M)=H^0(U_d\cap\Gamma_0(3,N),M)=M^{U_d\cap\Gamma_0(3,N)}.$$

By Theorem~\ref{U-inv}, $M^{U_d\cap\Gamma_0(n,N)}\isom M(c;a,b)_\chi^d$.  Recall that $M(c;a,b)$ is just $F(a,b)$ as a $\GL_2$-module, with an additional action of $\GL_1$ through the $c$-power map.
Hence, we are interested in finding a certain class in $H_1(\Gamma_0(2,N/d)^\pm,F(a,b)_{\chi_1})$ where $P_d$ acts on $F(a,b)_{\chi_1}$, (and hence on the cohomology) as described in Theorem~\ref{U-inv}.

We have chosen $f\in H_1(\Gamma_0(2,N_1)^\pm,F(a,b)_{\chi_1})$ to correspond to $\tau$, so that for the two dimensional Hecke operator $T_\ell=T(\ell,1)$, we have $f|_{T_\ell}=\lambda_\ell f$, where $\lambda_\ell=\tr(\tau(\frob_\ell))$.  We now examine how the three-dimensional Hecke operators $T(\ell,1)$ and $T(\ell,2)$ act on $f$.

For each prime $\ell\nmid pN$ and for each coset representative in $T(\ell,1)$, we apply Theorem~\ref{calcx} to translate the coset representatives $s$ by an element $\gamma\in\Gamma_0(3,N)$ into $P_d$ and then let 
$x=g_ds\gamma g_d^{-1}\in P_0$, and obtain:

If $s=\begin{pmatrix}1&0&0\cr0&1&0\cr \alpha_1&\alpha_2&\ell\end{pmatrix}$ then (by case 1 of Theorem~\ref{calcx}), 
$$\psi_d^1(s\gamma)=\psi_0^1(x)=1\text{ and }\psi_d^2(s\gamma)=\psi_0^2(x)=\begin{pmatrix}1&0\cr \alpha_2-\alpha_1d&\ell\end{pmatrix}.$$

If $s=\begin{pmatrix}1&0&0\cr \alpha_0&\ell&0\cr 0&0&1\end{pmatrix}$ with $\ell\nmid \alpha_0d+1$, then (by case 3 of Theorem~\ref{calcx})
$$\psi_0^1(x)=1\text{ and }\psi_0^2(x)=\begin{pmatrix}\ell&0\cr0&1\end{pmatrix}.$$

If $s=\begin{pmatrix}1&0&0\cr \alpha_0&\ell&0\cr 0&0&1\end{pmatrix}$ with $\ell|\alpha_0d+1$, then  (by case 4 of Theorem~\ref{calcx}) 
$$\psi_0^1(x)=\ell\text{ and }\psi_0^2(x)=\begin{pmatrix}1&0\cr0&1\end{pmatrix}.$$

If $s=\begin{pmatrix}\ell&0&0\cr 0&1&0\cr 0&0&1\end{pmatrix}$ then (by case 2 of Theorem~\ref{calcx})
$$\psi_0^1(x)=1\text{ and }\psi_0^2(x)=\begin{pmatrix}\ell&0\cr0&1\end{pmatrix}.$$

Define the $S$-sheaf on $C_0$ by $y\mapsto H_1(P_y\cap\Gamma_0(3,N)),M)$, where $P_y$ is the stabilizer of $y$.  Use Theorem~\ref{hecke-equ} and the fact that $N$ is assumed to be squarefree Then by Section~\ref{preparing}, we know that $s$ preserves the $\Gamma$-orbits of $C_0$ and $T(\ell,1)=\oplus T_{yy}$ where $y$ runs through a set of representatives of the $\Gamma$-orbits.  Since $f$ is supported on the orbit of $(1:d:0)$ and the $g_{yy}$ are just the $s\gamma$, we obtain from Theorems~\ref{hecke-equ} and \ref{U-inv}(3),
\begin{align*}
f|T(\ell,1)&=\sum_{\alpha_1,\alpha_2}f\bigg|_{\chi_1}\begin{pmatrix}1&0\cr \alpha_2-\alpha_1d&\ell\end{pmatrix}\cr&\quad+\sum_{\alpha_0:\ell\nmid \alpha_0d+1}f\bigg|_{\chi_1}\begin{pmatrix}\ell&0\cr0&1\end{pmatrix}+\chi_0(\ell)\ell^cf\bigg|_{\chi_1}I+f\bigg|_{\chi_1}\begin{pmatrix}\ell&0\cr0&1\end{pmatrix}
\cr
&=\ell\left(f\bigg|T_\ell-f\bigg|_{\chi_1}\begin{pmatrix}\ell&0\cr0&1\end{pmatrix}\right)+(\ell-1)f\bigg|_{\chi_1}\begin{pmatrix}\ell&0\cr0&1\end{pmatrix}+f\bigg|_{\chi_1}\begin{pmatrix}\ell&0\cr0&1\end{pmatrix}+\chi_0(\ell)\ell^cf\cr
&=(\ell \lambda_\ell +\chi_0(\ell)\ell^c)f.\end{align*}

Similarly, for $T(\ell,2)$:

If $s=\begin{pmatrix}1&0&0\cr \alpha_0&\ell&0\cr \alpha_1&0&\ell\end{pmatrix}$ with $\ell\nmid \alpha_0d+1$, then (by case 3 of Theorem~\ref{calcx}) $$\psi_0^1(x)=1\text{ and }\psi_0^2(x)=\begin{pmatrix}\ell&0\cr -\alpha_1\ell d&\ell\end{pmatrix}.$$

If $s=\begin{pmatrix}1&0&0\cr \alpha_0&\ell&0\cr \alpha_1&0&\ell\end{pmatrix}$ with $\ell|\alpha_0d+1$, then (by case 4 of Theorem~\ref{calcx}) $$\psi_0^1(x)=\ell\text{ and }\psi_0^2(x)=\begin{pmatrix}1&0\cr -\alpha_1 d&\ell\end{pmatrix}.$$

If $s=\begin{pmatrix}\ell&0&0\cr 0&1&0\cr 0&\alpha_2&\ell\end{pmatrix}$  then (by case 2 of Theorem~\ref{calcx}) $$\psi_0^1(x)=1\text{ and }\psi_0^2(x)=\begin{pmatrix}\ell&0\cr \alpha_2\ell&\ell\end{pmatrix}.$$

If $s=\begin{pmatrix}\ell&0&0\cr 0&\ell&0\cr 0&0&1\end{pmatrix}$  then (by case 1 of Theorem~\ref{calcx})  
$$\psi_0^1(x)=\ell\text{ and }\psi_0^2(x)=\begin{pmatrix}\ell&0\cr 0&1\end{pmatrix}.$$

Hence,
\begin{align*}
f|T(\ell,2)&=\sum_{\alpha_0,\alpha_1:\ell\nmid \alpha_0d+1}f\bigg|_{\chi_1}\begin{pmatrix}\ell&0\cr0&\ell\end{pmatrix}+\sum_{\alpha_1}\chi_0(\ell)\ell^cf\bigg|_{\chi_1}\begin{pmatrix}1&0\cr -\alpha_1 d&\ell\end{pmatrix}\cr&\qquad\qquad+\sum_{\alpha_2}f\bigg|_{\chi_1}\begin{pmatrix}\ell&0\cr0&\ell\end{pmatrix}+\chi_0(\ell)\ell^cf\bigg|_{\chi_1}\begin{pmatrix}\ell&0\cr 0&1\end{pmatrix}
\cr
&=\ell^2f\bigg|_{\chi_1}\begin{pmatrix}\ell&0\cr0&\ell\end{pmatrix}+\chi_0(\ell)\ell^cf\bigg|T_\ell\cr
&=(\chi_1(\ell)\ell^{a+b+2}+\chi_0(\ell)\ell^c\lambda_\ell)f.
\end{align*}

Therefore $f$ is attached to $\omega\tau\oplus\omega^c\chi_0=\rho$, which has predicted weight $F(a,b,c)$.  Since $f$ is not attached to a direct sum of characters, by Theorem~\ref{d} it not only appears in $E_{0,3}^1$, but also survives into $E_{0,3}^\infty$, and hence appears in $H_3(\Gamma_0(3,N),F(a,b,c)_\chi)$.  We have thus proved the following theorem:

\begin{theorem}\label{main} Let $p>3$ and let $\rho:G_\Q\to\GL(3,\overline{ \mathbb F}_p)$ be a Galois representation of squarefree conductor $N$ and nebentype $\chi$ that decomposes as a sum of a character and an irreducible odd two-dimensional Galois representation $\tau$. Then $\rho$ is attached to a cohomology eigenclass in $H_3(\Gamma_0(3,N),V_\chi)$, where $V$ is the first of the two weights predicted by conjecture~\ref{conjecture}.
\end{theorem}

We note that the second of the two weights predicted by Conjecture~\ref{conjecture} can also be shown to work.  Let ${}^t\rho^{-1}\otimes\omega^2$ be the twisted contragredient of $\rho$.  This representation will still be a sum of a two-dimensional, odd, irreducible representation and a character, and as such, will be attached to an eigenclass for its first predicted weight of Conjecture~\ref{conjecture}, by Theorem~\ref{main}.  Since Conjecture~\ref{conjecture} is compatible with duality according to the prescription in \cite[Proposition 2.8]{AS}, we find that $\rho$ is attached to an eigenclass in the dual of this weight.  A simple computation shows that this dual is exactly the second predicted weight for $\rho$.  Hence, $\rho$ is attached to eigenclasses in both of the weights described in Conjecture~\ref{conjecture}, concluding the proof that Conjecture~\ref{conjecture} is true for representations of squarefree conductor. 

\section{Appendix: Proof of Theorem~\ref{calcx}}

\begin{proof} To prove case 1, we take $\gamma=I$, since $s$ is already in $P_d$.  Hence $x=g_dsg_d^{-1}$, and we find that $x_{11}=\ell_1$, and $\psi_0^2(x)=\begin{pmatrix}\ell_2&0\cr c-bd&\ell_3\end{pmatrix}$.

For the other cases, note that a matrix is in $P_0g_d$ if and only if it is of the form
$$\begin{pmatrix}r&rd&0\cr *&*&*\cr *&*&*\end{pmatrix}.$$

Since we want
 $$g_ds\gamma=\begin{pmatrix}A(\ell_1+ad)+C\ell_2d&B(\ell_1+ad)+D\ell_2d&0\cr *&*&0\cr *&*&*\end{pmatrix}\in P_0g_d,$$
we see that we must have $B(\ell_1+ad)+D\ell_2d=d(A(\ell_1+ad)+C\ell_2d)$.  Writing this as a matrix equation, we have that
$$(\ell_1+ad,d\ell_2)\begin{pmatrix}A&B\cr C&D\end{pmatrix}=r(1,d)$$ for some $r$.  We will assume that $r>0$ (changing the signs of $A,B,C,D$ if needed).  Since the matrix $\begin{pmatrix}A&B\cr C&D\end{pmatrix}$ must have determinant 1, we get
$$(\ell_1+ad,d\ell_2)=r(1,d)\begin{pmatrix}D&-B\cr -C&A\end{pmatrix}.$$

We now specialize to case 2, in which $\ell_1=\ell$ is prime, $\ell_2=1$ and $a=0$.  Then we have
$$(\ell,d)=r(1,d)\begin{pmatrix}D&-B\cr -C&A\end{pmatrix}.$$
Since $(\ell,d)$ and $(1,d)$ are both primitive, we must have $r=1$.

Multiplying, we see that 
$$\ell=D-Cd,\text{ and }d=-B+Ad$$
Solving for $A$ and $D$ (and using that $B=kN$ for some $k\in\Z$ since $\gamma\in\Gamma_0(N)$) we obtain
$$A=1+\frac{kN}d,\quad B=kN,\quad D=\ell+Cd,$$
and
$$1=AD-BC=\left(1+\frac{kN}d\right)(\ell+Cd)-kNC=\left(\frac{kN}d+1\right)\ell+Cd$$
Now, using the Chinese Remainder Theorem to choose $C$ so that
$$Cd\equiv 1\pmod\ell\text{ and }Cd\equiv 1-\ell\pmod{N/d}$$
we may choose 
$$k=\frac{-Cd-\ell+1}{(N/d)\ell}$$ and we have our desired 
$$\gamma=\begin{pmatrix}A&kN\cr C&D\end{pmatrix}.$$

Note that $x_{11}=r$, and by determinant considerations, we must have $x_{11}x_{22}=\ell_1\ell_2$.  Hence, in case 2, $x_{11}=1$, $x_{22}=\ell$, and a quick calculation shows that
$$\psi_0^2(x)=\begin{pmatrix}\ell&0\cr c\ell-bd&\ell_3\end{pmatrix}.$$

In case 3, we proceed similarly.  We have
$$(1+ad,\ell d)=r(1,d)\begin{pmatrix}D&-B\cr-C&A\end{pmatrix},$$
with both $(1+ad,\ell d)$ and $(1,d)$ primitive, so that $r=1$.  Hence $ad+1=D-Cd$ and $\ell d=-B+Ad$.  We set $B=kN$ and solve for $A=k(N/d)+\ell$ and $D=(ad+1)+Cd$.

Now, using the fact that we want 
$$1=AD-BC=((ad+1)N/d)k+C(\ell d)+(ad+1)\ell,$$
and the fact that $((ad+1)N/d,\ell d)=1$, we see that integers $C$ and $k$ exist, so $\gamma$ exists.  We compute $x$ and obtain $\psi_0^1(x)=r=1$ and $\psi_0^2(x)=\begin{pmatrix}\ell&0\cr c(ad+1)-b\ell d&\ell_3\end{pmatrix} $.

Finally, for case 4, we again proceed in a similar fashion.  We find that
$$(1+ad,\ell d)=r(1,d)\begin{pmatrix}D&-B\cr-C&A\end{pmatrix}$$

Under the assumption that a $\gamma$ exists, $r=\ell$ (since $(1,d)$ is primitive, but the GCD $(1+ad,\ell d)=\ell$).  Hence, computing $x$ using the values for $A,B,C,D$ below, we find that $\psi_0(x)=\ell$ and $\psi_0^2(x)=\begin{pmatrix}1&0\cr c\left(\frac{1+ad}\ell\right)-bd&\ell_3\end{pmatrix}$.

To show the existence of $\gamma$, we note that $\ell(D-Cd)=1+ad$ and $\ell(-B+Ad)=\ell d$.  Setting $B=kN$ and solving, we find that
$$A=1+k\frac{N}d,\quad B=kN,\quad D=\frac{1+ad}\ell+Cd.$$
Then 
$$1=AD-BC=\frac{1+ad}\ell+k\left(\frac Nd\frac{(1+ad)}\ell\right)+Cd.$$
Since $d$ and $\left(\frac Nd\frac{(1+ad)}\ell\right)$ are relatively prime, a solution exists for $k$ and $C$.
\end{proof}

\end{document}